\documentclass[11pt,amssymb]{amsart}
\usepackage{amssymb}
\usepackage[mathscr]{eucal}
\usepackage[all,cmtip]{xy}

\newcommand{\im}{{\rm im}\:}

\newcommand{\Q}{{\mathbb Q}}

\newcommand{\Gal}{\mathrm{Gal}}

\newcommand{\Aut}{\mathrm{Aut}}

\newtheorem{thm}{Theorem}[section]
\newtheorem{lemma}[thm]{Lemma}
\newtheorem{prop}[thm]{Proposition}
\newtheorem{cor}[thm]{Corollary}

.240pk scaled 1200 .240pk
\begin{document}
\title[Bounded generation]{Non-virtually abelian anisotropic linear groups are not boundedly generated}

\author[P.~Corvaja]{Pietro Corvaja}
\author[A.~Rapinchuk]{Andrei S. Rapinchuk}
\author[J.~Ren]{Jinbo Ren}
\author[U.~Zannier]{Umberto M. Zannier}

\address{Dipartimento di Scienze Matematiche, Informatiche e Fisiche, via delle Scienze,
206, 33100 Udine, Italy}

\email{pietro.corvaja@uniud.it}

\address{Department of Mathematics, University of Virginia,
Charlottesville, VA 22904-4137, USA}

\email{asr3x@virginia.edu}

\address{Department of Mathematics, University of Virginia (formerly affiliated),
Charlottesville, VA 22904-4137, USA}

\email{renjinbomath@gmail.com}

\address{Scuola Normale Superiore, Piazza dei Cavalieri, 7, 56126 Pisa, Italy}

\email{umberto.zannier@sns.it}


\begin{abstract}
We prove that if a linear group $\Gamma \subset \mathrm{GL}_n(K)$ over a field $K$ of characteristic zero is boundedly generated by semi-simple
(diagonalizable) elements then it is virtually solvable. As a consequence, one obtains that infinite $S$-arithmetic subgroups of absolutely almost simple anisotropic algebraic groups over number fields are {\it never} boundedly generated. Our proof relies on Laurent's theorem from Diophantine geometry and properties of generic elements.

\end{abstract}

\maketitle

\section{Introduction}

An abstract group $\Gamma$ has the property of {\bf bounded generation} (BG) if there exist elements $\gamma_1, \ldots , \gamma_r \in \Gamma$ (not necessarily distinct) such that
\begin{equation}\tag{BG} \label{E:BG}
\Gamma = \langle \gamma_1 \rangle \cdots \langle \gamma_r \rangle,
\end{equation}
where $\langle \gamma \rangle$ denotes the cyclic subgroup of $\Gamma$ generated by an element $\gamma \in \Gamma$. (In this case, we also say that $\Gamma$ is {\bf boundedly generated} by $\gamma_1, \ldots , \gamma_r$.) Obviously, every group satisfying (BG) is finitely generated. Conversely, every finitely generated (virtually) abelian, or more generally (virtually) nilpotent, group has bounded generation - note that such a group is also linear. On the other hand, the fact that there exist {\bf non-virtually solvable} linear groups having (BG) is rather nontrivial. Historically, the first examples emerged as a consequence of the result of D.~Carter and G.~Keller \cite{CK} that for any $n \geq 3$ and any ring $\mathcal{O}$ of algebraic integers, every unimodular $(n \times n)$-matrix with entries in $\mathcal{O}$  is a product of a {\bf bounded} number (with a bound depending on $n$ and the discriminant of $\mathcal{O}$) of elementary matrices over $\mathcal{O}$. Indeed, this fact immediately implies the existence for $\Gamma = \mathrm{SL}_n(\mathcal{O})$ of a presentation (\ref{E:BG}) in which the $\gamma_i$'s are suitable elementary (in particular, unipotent) matrices. Although $\mathrm{SL}_2(\mathcal{O})$ does not have (BG) when $\mathcal{O}$ is either $\mathbb{Z}$ or the ring of integers of an imaginary quadratic field\footnote{This follows from the fact that the group $\mathrm{SL}_2(\mathcal{O})$ is virtually free in the first case and virtually has a nonabelian free quotient in the second \cite{GrSch}.}, it does have (BG) in all other arithmetic situations, i.e. when $\mathcal{O}$ is the ring of $S$-integers of a number field with infinite unit group $\mathcal{O}^{\times}$. Again, this is derived from the fact that in the situation at hand every matrix in $\mathrm{SL}_2(\mathcal{O})$ is a product of at most 9 elementary matrices \cite{MRS}, and then the resulting presentation (BG) involves a mix of unipotent and semi-simple (diagonalizable) matrices, with the unipotent ones definitely present. The result of \cite{CK} was extended by O.I.~Tavgen \cite{Tav} to all Chevalley groups of rank $> 1$, and also to most quasi-split groups. Bounded generation of $S$-arithmetic subgroups in isotropic, but not necessarily quasi-split, orthogonal groups of quadratic forms over number fields was established (under some natural assumptions) in \cite{ER}.  These (and some other) results generated the expectation that  higher rank $S$-arithmetic subgroups of absolutely almost simple algebraic groups over number fields should have (BG).

To put this problem into perspective, we recall that while being a purely combinatorial property of groups, bounded generation has a number of interesting consequences and applications in different areas. First, a group having (BG) in which every finite index subgroup has finite abelianization is {\bf $SS$-rigid}, i.e. has only finitely many inequivalent completely reducible complex representations in each dimension \cite[Appendix A]{PR}. Second, it was shown in \cite{Lub} and \cite{PR-CSP} that $S$-arithmetic subgroups of absolutely almost simple algebraic groups that have (BG) under some natural assumptions possess the {\bf congruence subgroup property} in the sense that the corresponding congruence kernel is finite. Third, property (BG) played a crucial role in the proof of the Margulis-Zimmer conjecture on  commensurated subgroups for higher rank $S$-arithmetic subgroups of Chevalley groups \cite{ShWil}, estimation of Kazhdan constants \cite{Kas}, \cite{Shal} and other situations. For the sake of completeness, we also mention that the natural analog of (BG) is of significance in the theory of profinite groups; in particular, the pro-$p$ groups with bounded generation are precisely the $p$-adic analytic groups \cite{Dixon}. Thus, bounded generation has long been regarded as an abstract property that can potentially provide a uniform approach to some important problems for $S$-arithmetic subgroups including  Serre's congruence subgroup and the Margulis-Zimmer conjectures, and also explain some rigidity phenomena. This would be particularly valuable in the  anisotropic case where many results involving arithmetic groups often rely on {\it ad hoc} techniques. We recall that a semi-simple algebraic group $G$ over a field $K$ of characteristic zero is $K$-anisotropic (i.e., has $K$-rank zero) if and only if the group $G(K)$ contains no nontrivial unipotent elements \cite{BT},  hence consists entirely of semi-simple elements. Building on this characterization, we will  call {\bf anisotropic} any  subgroup $\Gamma \subset \mathrm{GL}_n(K)$ (in other words, any linear group) that contains only semi-simple elements. It should be noted that while quite a few examples of $S$-arithmetic subgroups of absolutely almost simple algebraic groups that possess (BG) have been found over the years, none of these was anisotropic, which brings us to the following.

\vskip2mm

\noindent {\bf Question A.} {\it Can $(\mathrm{BG})$ possibly hold for an infinite $S$-arithmetic subgroup of  an anisotropic absolutely almost simple algebraic group?}

\vskip2mm

To approach this question, one may first try to re-examine the nature of presentations (\ref{E:BG}) that arise in the known examples of boundedly generated groups. As we pointed out above, in the presentations (\ref{E:BG}) for $\mathrm{SL}_n$, $n \geq 2$, over the rings of algebraic $S$-integers in appropriate situations that are derived from bounded generation of these groups by elementary matrices, some or even all elements $\gamma_i$ are unipotent. So, one may wonder if in these (or some other) examples one can produce a presentation (\ref{E:BG}) with all the $\gamma_i$'s being semi-simple -- in which case we would say that $\Gamma$ is boundedly generated by semi-simple elements. Along these lines, one can ask the following general question which in a way subsumes Question A.

\vskip2mm

\noindent {\bf Question B.} {\it Which linear groups  are boundedly generated by semi-simple elements?}

\vskip2mm

The goal of this paper is to give, in the case of arbitrary linear groups over a field of characteristic zero, a powerful necessary condition for bounded generation by semi-simple elements. This condition, in particular, leads to a negative answer to Question A.
\begin{thm}\label{mainthm}
Let $\Gamma \subset \mathrm{GL}_n(K)$ be a linear group over a field $K$ of characteristic zero, which is not virtually solvable. Then in any possible presentation {\rm (\ref{E:BG})} for $\Gamma$ at least two of the elements $\gamma_i$ must be non-semi-simple. In particular, a linear group over a field of characteristic zero boundedly generated by semi-simple elements is virtually solvable.
\end{thm}

There exist virtually solvable finitely generated linear groups that do not admit bounded generation by any elements - semi-simple or not (see Example \ref{zhujie}.1), so Theorem \ref{mainthm} is {\bf not} a criterion. However, it yields the following criterion in the case of anisotropic groups.

\begin{cor}\label{C:1}
An anisotropic linear group $\Gamma \subset \mathrm{GL}_n(K)$ over a field of characteristic zero has {\rm (BG)} if and only if it is finitely generated and virtually abelian.
\end{cor}

To formulate our result for $S$-arithmetic groups, we need to introduce one additional notation. Let $G$ be a linear algebraic group defined over a number field $K$, and let $S$ be a finite set of valuations of $K$ containing all archimedean ones. We set
$$
G_S := \prod_{v \in S} G(K_v),
$$
where $K_v$ denotes the completion of $K$ with respect to $v$, and recall that for $G$ semi-simple, the non-compactness of $G_S$ is equivalent to the fact that the $S$-arithmetic subgroups of $G$ are infinite (cf. \cite[\S 5.4]{PR}).

\begin{thm}\label{zdlx}
Let $G$ be an algebraic group over a number field $K$, and let $S$ be a finite set of valuations of $K$ containing all archimedean ones. Assume that the quotient $G^{\circ}/R$ of the connected component of $G$ by its radical possesses a $K$-defined semi-simple $K$-anisotropic normal subgroup $H$ such that the group $H_S$ is non-compact (which automatically holds if $G^{\circ}$ itself possesses such a subgroup). Then the $S$-arithmetic subgroups of $G$ are \underline{not} boundedly generated. In particular, infinite $S$-arithmetic subgroups of absolutely almost simple $K$-anisotropic groups are \underline{not}  boundedly generated.
\end{thm}

The notion of bounded generation has a profinite version. More precisely, a profinite group $\Delta$ has the property of bounded generation $(\mathrm{BG})_{\mathrm{pr}}$ as a profinite group if there exist elements $\delta_1, \ldots , \delta_r \in \Delta$ (not necessarily distinct) such that
$$
\Delta = \overline{\langle \delta_1 \rangle} \cdots \overline{\langle \delta_r \rangle},
$$
where $\overline{\langle \delta_i \rangle}$ is the closure of the cyclic subgroup generated by $\delta_i$. If a (discrete) group $\Gamma$ has property (BG), i.e. admits a factorization (BG) as above, then its profinite completion $\widehat{\Gamma}$ inherits the following factorization
$$
\widehat{\Gamma} = \overline{\langle \gamma_1 \rangle} \cdots \overline{\langle \gamma_r \rangle},
$$
and hence has property $(\mathrm{BG})_{\mathrm{pr}}$ of bounded generation as a profinite group. The question of whether the converse is true, i.e. whether $(\mathrm{BG})_{\mathrm{pr}}$ for $\widehat{\Gamma}$ implies $(\mathrm{BG})$ for $\Gamma$, remained open for a long time. Combining Theorem \ref{zdlx} with the known results on the congruence subgroup problem, one obtains the negative answer to this question.

\begin{cor}\label{negative}
There exist residually finite finitely generated groups $\Gamma$  that do {\bf not} have property $(\mathrm{BG})$ of bounded generation but whose profinite completion $\widehat{\Gamma}$ does have  property $(\mathrm{BG})_{\mathrm{pr}}$ of bounded generation as a profinite group.
\end{cor}

(We note that our construction produces such groups $\Gamma$ that are actually $S$-arithmetic subgroups of absolutely almost simple algebraic groups defined over number fields where $S$ is a finite set of valuations of the base field containing all archimedean ones.)

Now, we would like to make a remark about the methods used to prove these results. One general technique that has been used to show that various groups, including free amalgamated products and HNN-extensions subject to certain conditions\footnote{See \S \ref{zhujie} for a precise formulation in the case of free amalgamated products.}, lattices in rank one groups etc., do not have bounded generation involves bounded cohomology, cf. \cite{F1}-\cite{F3},  \cite{Gr1}-\cite{Gr2}. More precisely, one observes that if a group is boundedly generated then its second bounded cohomology is a finite-dimensional real vector space, and then disproves bounded generation by showing that in the cases of interest this space is actually infinite dimensional. Unfortunately, this approach cannot be used to show the absence of bounded generation in the higher rank $S$-arithmetic subgroups of anisotropic absolutely almost simple algebraic groups as for these groups the second bounded cohomology vanishes \cite{BuMo}. Instead, our method hinges on the results from Diophantine geometry (more specifically, Laurent's theorem -- see Theorem \ref{dio} below) and uses the existence of generic elements in Zariski-dense subgroups \cite{PrR1}-\cite{PrR3}. The application of these techniques, particularly in such combination, in the context of group theory appears to be novel, so it would be interesting to see whether these can be employed to tackle some other group-theoretic problems.

Finally, bounded generation of $\mathrm{SL}_n({\mathcal O})$ by elementaries over the ring
${\mathcal O}$ of $S$-integers of a number field when either $n \geq 3$ or $n = 2$ and
the group of units ${\mathcal O}^{\times}$ is infinite can be used to construct a {\it polynomial
parametrization} of $\mathrm{SL}_n({\mathcal O})$ in those cases. This means that there exists
a polynomial matrix $A(x_1, \ldots , x_r) \in \mathrm{SL}_n(\mathbb{Z}[x_1, \ldots , x_r])$ such that
the values $A(a_1, \ldots , a_r)$ with $a_i \in {\mathcal O}$ fill up all of $\mathrm{SL}_n({\mathcal O})$.
It is rather remarkable, however, that a polynomial parametrization of $\mathrm{SL}_2({\mathcal O})$
exists even when the group does not have bounded generation, i.e. when ${\mathcal O}$ is either $\mathbb{Z}$
or the ring of integers of an imaginary quadratic field - see \cite{Larsen}, \cite{Vasser}. Likewise, bounded generation of $\mathrm{SL}_n({\mathcal O})$ by semi-simple elements would give its parametrization by purely exponential polynomials. Since now we know that bounded generation by of this group by semi-simples is impossible, it would be interesting to determine if $\mathrm{SL}_n({\mathcal O})$ in all or at least some cases can still be parametrized by purely exponential polynomials.

The structure of the paper is the following. In \S \ref{special}, we show that it is enough to prove Theorem \ref{mainthm} for a subgroup $\Gamma \subset \mathrm{GL}_n(K)$ where $K$ is a number field. In \S \ref{gpreduction} we recall a result on the existence of generic elements in Zariski-dense subgroups and  derive consequences needed for our argument. In \S4, we use Laurent's theorem to establish a key statement about finite product of cyclic subgroups in $\mathrm{GL}_n(K)$ (where $K$ is a number field) generated by elements which are semi-simple with one possible exception -- see Theorem \ref{T:notsubset}.
We then prove Theorems \ref{mainthm}, \ref{zdlx} and Corollaries \ref{C:1} and \ref{negative} in \S \ref{S:main}. Finally, in \S \ref{zhujie} we give an example of a finitely generated linear solvable group without (BG), make concluding remarks and formulate some open problems.

\section{A reduction to linear groups over number fields}\label{special}

The goal of this section is to reduce the proof of Theorem \ref{mainthm} to the case of linear groups over number fields. The argument is based on the following proposition that enables us to construct a suitable specialization.

\begin{prop}\label{specialization}
Let $R$ be a finitely generated $\mathbb{Q}$-algebra without zero divisors. Given a non-virtually solvable subgroup
$\Gamma \subset \mathrm{GL}_n(R)$ and semi-simple elements $\gamma_1, \ldots , \gamma_r \in \Gamma$, there exists a $\mathbb{Q}$-algebra homomorphism $\theta \colon R \to F$ to a number field $F$ such that for the corresponding group homomorphism $\Theta \colon \mathrm{GL}_n(R) \to \mathrm{GL}_n(F)$, the image
$\Theta(\Gamma)$ is not a virtually solvable group and each of the elements $\Theta(\gamma_1), \ldots , \Theta(\gamma_r)$ is semi-simple.
\end{prop}

We note that  similar statements but without the assertion of the semi-simplicity of the images of the given semi-simple elements can be found in
\cite[\S 3]{Ohinvent} and \cite[Proposition 16.4.13]{gpgrowth}. For the proof we need the following lemma. Given a group $\Delta$, we let $\mathscr{D}^m(\Delta)$ denote the $m$-th term of the derived series of $\Delta$, and for $\ell \in \mathbb{N}$ let $\Delta^{(\ell)}$ denote the (normal) subgroup of $\Delta$ generated by the $\ell$-th powers of its elements.
\begin{lemma}\label{L:1}
There exists $\ell = \ell(n)\in \mathbb{N}$ such that for \underline{any} virtually solvable subgroup $\Delta \subset \mathrm{GL}_n(E)$, where $E$ is a field of characteristic zero, we have
$$
\mathscr{D}^{n+1}(\Delta^{(\ell)}) = \{ I_n \}.
$$
\end{lemma}
\begin{proof}
We can assume without loss of generality that $E$ is algebraically closed.
Let $G$ be the Zariski-closure of $\Delta$. Since $\Delta$ is virtually solvable, the connected component $G^{\circ}$ is solvable, hence triangularizable by the  Lie-Kolchin theorem (cf. \cite[Corollary 10.5]{borel}). It follows that
\begin{equation}\label{E:Z1}
\mathscr{D}^n(G^{\circ}) = \{ I_n \}.
\end{equation}
On the other hand, according to Bass's generalization of Jordan's theorem \cite[Theorem 1]{Bass}, there exists $j \in \mathbb{N}$ depending only on $n$ such that
$G/G^{\circ}$ has an abelian normal subgroup $H$ of index at most $j$. Set $\ell = j!$. Then for the canonical morphism $f \colon G \to G/G^{\circ}$ we  have
$$
f(\Delta^{(\ell)}) = f(\Delta)^{(\ell)} \subset H,
$$
and therefore $f(\mathscr{D}^1(\Delta^{(\ell)})) = \{ e \}$, i.e. $\mathscr{D}^1(\Delta^{(\ell)}) \subset G^{\circ}$. Combining this with (\ref{E:Z1}), we obtain our claim.
\end{proof}

Beginning the proof of Proposition \ref{specialization}, let us first show that
\begin{equation}\label{E:1}
\mathscr{D}^{n+1}(\Gamma^{(\ell)}) \neq \{ I_n \},
\end{equation}
where $\ell = \ell(n)$ is the constant from Lemma \ref{L:1}.
Assume the contrary, and let $G$ and $H$ denote the Zariski-closures of $\Gamma$ and $\Gamma^{(\ell)}$, respectively; clearly, $H$ is a normal subgroup of $G$. For the power map $$\mu \colon G \to G, \ \  g \mapsto g^{\ell},$$ we have $\mu(\Gamma) \subset H$, and therefore $\mu(G) \subset H$. This means that the quotient $G/H$ is an algebraic group of the finite exponent $\ell$, hence finite since $\mathrm{char}\: E = 0$. On the other hand, since $D^{n+1}(\Gamma^{(\ell)}) = \{ I_n \}$, the group $H$ is solvable. This means that the finite index subgroup $\Gamma \cap H \subset \Gamma$ is solvable, making $\Gamma$ virtually solvable. This contradicts our assumption proving (\ref{E:1}).

 According to (\ref{E:1}), we can pick an element $g \in \mathscr{D}^{n+1}(\Gamma^{(\ell)}) \setminus \{ I_n \}$, and let $a$ be any nonzero entry of the matrix $g - I_n$. Next, let $L$ be the field of fractions of $R$, and let $f_i(t) \in L[t]$ will be the minimal polynomial of the matrix $\gamma_i \in \mathrm{GL}_n(R)$. Replacing $R$ with a larger finitely generated $\mathbb{Q}$-subalgebra $R' \subset L$, we can assume that $f_i \in R[t]$. Since $\gamma_i$ is semi-simple, the polynomial $f_i$ does not have multiple roots, hence its discriminant $d_i$ is $\neq 0$. Set $d = d_1 \cdots d_r$. Replacing $R$ by the localization $R\left[ \frac{1}{ad} \right]$, we can assume that both $a$ and $d$ are invertible in $R$. Let $\mathfrak{m}$ be a maximal ideal of $R$. Since $R$ is a finitely generated $\mathbb{Q}$-algebra, it follows from a version of Nullstellensatz (cf. \cite[Ch. IX, Corollary 1.2]{Langalgebra}) that $F := R/\mathfrak{m}$ is a finite extension of $\mathbb{Q}$. We will now show that the canonical homomorphism $\theta \colon R \to F$ is as required.

First, since $a \in R^{\times}$, we have $\theta(a) \neq 0$, and therefore $\Theta(g) \neq I_n$. Thus, for $\Delta = \Theta(\Gamma)$ we have
$$
\mathscr{D}^{n+1}(\Delta^{(\ell)}) = \Theta(\mathscr{D}^{n+1}(\Gamma^{(\ell)})) \neq \{ I_n \},
$$
which by Lemma \ref{L:1} implies that $\Delta$ is not virtually solvable. Second, since $d \in R^{\times}$, for each $i = 1, \ldots , r$ we have $\theta(d_i) \neq 0$, which means that the (monic) polynomial $\bar{f}_i(t) \in F[t]$, obtained by applying $\theta$ to the coefficients of $f_i$, does not have multiple roots. Since $\bar{f}_i(\Theta(\gamma_i)) = 0$, the minimal polynomial of the matrix $\Theta(\gamma_i)$ does not have multiple roots either, implying that this matrix is semi-simple, as required.

\vskip2mm

\noindent {\bf Reduction 2.3.} {\it If the first assertion of Theorem \ref{mainthm} is valid for all subgroups $\Delta \subset \mathrm{GL}_n(F)$ where $F$ is a number field (i.e., every such subgroup that has a presentation (\ref{E:BG}) in which all elements $\gamma_i$, with one possible exception, are semi-simple is necessarily virtually solvable) then it is also valid for all subgroups $\Gamma \subset \mathrm{GL}_n(K)$ where $K$ is any field of characteristic zero.}

\vskip2mm

Indeed, assume that there is a non-virtually solvable subgroup $\Gamma \subset \mathrm{GL}_n(K)$ that admits a presentation (BG) in which all matrices $\gamma_i$, with one possible exception, are semi-simple. Let $R$ be the $\mathbb{Q}$-subalgebra generated by the entries of all the $\gamma_i$'s and their inverses; clearly $\Gamma \subset \mathrm{GL}_n(R)$. Then Proposition \ref{specialization} yields a homomorphism $\theta \colon R \to F$ to a number field $F$ such that for the corresponding homomorphism $\Theta \colon \mathrm{GL}_n(R) \to \mathrm{GL}_n(F)$, the group $\Delta = \Theta(\Gamma)$ is not virtually solvable and all matrices $\Theta(\gamma_1), \ldots , \Theta(\gamma_r)$, with one possible exception, are semi-simple. Applying $\Theta$ to (BG), we obtain
$$
\Delta = \langle \Theta(\gamma_1) \rangle \cdots \langle \Theta(\gamma_r) \rangle.
$$
Then by our assumption $\Delta$ must be virtually solvable, which is not the case by our construction. A contradiction, justifying the reduction.

\section{On generic elements and their eigenvalues}\label{gpreduction}

The proof of Theorem \ref{mainthm} requires a result (see Proposition \ref{gp} below) stating that given a finite set of elements in a linear group over a field of characteristic zero whose Zariski-closure has semi-simple connected component, the group always contains a semi-simple element with the eigenvalues multiplicatively independent from those of the given elements. This fact is interesting in its own right, and its proof relies on the existence and properties of \textbf{generic elements}, which we will now recall. (We refer the reader to \cite{borel} for the notions related to the theory of algebraic groups.)

Let $G$ be a semi-simple algebraic group over a field $K$, and let $T$ be a maximal $K$-torus of $G$. The absolute Galois group $\mathcal{G} = \Gal(K^{\mathrm{sep}}/K)$ naturally acts on the character group $X(T)$, and this action leaves the corresponding root system $\Phi := \Phi(G , T)$ invariant,
yielding a (continuous) group homomorphism
$$\rho_T\colon\mathcal{G}\longrightarrow \Aut(\Phi)$$
to the automorphism group of $\Phi$.

\addtocounter{thm}{1}

\vskip2mm

\noindent {\bf Definition 3.1.} {\it A maximal $K$-torus $T$ is said to be \textbf{generic} over $K$ (or $K$-generic) if the image $\im \rho_T$ contains the Weyl group $W(\Phi) \subset \Aut(\Phi)$. Furthermore, a regular semi-simple element $\gamma\in G(K)$ is  $K$-\textbf{generic} if the $K$-torus $T :=Z_G(\gamma)^{\circ}$ (``connected centralizer'') is generic over $K$.}

\vskip2mm

\noindent We refer the reader to \cite[\S 9]{PrR2} for a discussion of these notions.  The following existence theorem is proved in \cite{PrR1} (see also \cite[Theorem 9.6]{PrR2}); its extension to Zariski-dense subgroups of absolutely almost simple algebraic groups over fields of positive characteristic is given in \cite{PrR3}.

\begin{thm}\label{generic}
Let $G$ be a semi-simple algebraic group over a finitely generated field $K$ of characteristic zero\footnotemark, and let $\Gamma\subset G(K)$ be a finitely generated Zariski-dense subgroup. Then $\Gamma$ contains a $K$-generic semi-simple element without components of finite order.
\end{thm}

\footnotetext{I.e., a finitely generated extension of $\Q$.}

Here the \textbf{components} of an element are understood in terms of the decomposition $G=G_1 \cdots G_r$ as an almost direct product of absolutely almost simple groups. We note that a $K$-generic regular semi-simple element without components of finite order generates a Zariski-dense subgroup of a maximal torus that contains it, making this torus unique, cf. \cite[p. 22]{PrR3}. We also note that

\vskip1mm

\begin{lemma}\label{tezheng}
Let $G$ be a semi-simple algebraic group over a finitely generated field $K$ of characteristic zero, and let $\Gamma\subset G(K)$ be a finitely generated Zariski dense subgroup.
Then for any finitely generated extension $L$ of $K$, there exists a regular semi-simple element $\gamma \in \Gamma$ without components of finite order that satisfies the following condition: for the torus $T :=Z_G(\gamma)^{\circ}$ and any character $\chi \in X(T)$, the fact that $\chi(\gamma) \in L^{\times}$ implies $\chi(\gamma)=1$.
\end{lemma}

\begin{proof}
Since $L$ is itself a finitely generated field of characteristic zero, we can use Theorem \ref{generic} to find a regular semi-simple element $\gamma \in \Gamma$ without components of finite order that is $L$-generic. Let $T = Z_G(\gamma)^{\circ}$, and let $\chi \in X(T)$ be a character such that $\chi(\gamma) \in L^{\times}$. Then for any $\sigma \in \Gal(\overline{L}/L)$ we have
$$
(\sigma(\chi))(\gamma)=\sigma(\chi(\sigma^{-1}(\gamma)))=\chi(\gamma).
$$
As we observed earlier, the cyclic group $\langle \gamma\rangle$ is Zariski-dense in $T$, so the above equation yields $\sigma(\chi) = \chi$. On the other hand, the fact that $T$ is $L$-generic implies that $X(T)$ does not contain any nontrivial $\Gal(\overline{L}/L)$-fixed elements. Thus, $\chi = 0$  and $\chi(\gamma)=1$.
\end{proof}

For $\gamma \in \mathrm{GL}_n(K)$, we let $\Lambda(\gamma)$ denote the subgroup of $\overline{K}^{\times}$ generated by all eigenvalues of $\gamma$ in $\overline{K}^{\times}$. We observe that if $\gamma$ lies in $T(K)$ for some $K$-torus $T \subset \mathrm{GL}_n$, then $\Lambda(\gamma)$ coincides with the set of all values $\chi(\gamma)$ of the characters $\chi \in X(T)$.

\vskip1.5mm

\addtocounter{thm}{1}

\noindent {\bf Definition 3.4.} {\it Let $\gamma, \gamma_1, \ldots , \gamma_r \in \mathrm{GL}_n(K)$. We say that the eigenvalues of $\gamma$ are {\bf multiplicatively independent} from those of $\gamma_1, \ldots , \gamma_r$ if
\begin{equation}\label{E:Y1}
\Lambda(\gamma) \: \cap \: \left[\Lambda(\gamma_1) \cdots \Lambda(\gamma_r)\right] = \{ 1 \}.
\end{equation}}

%\vskip.5mm

\begin{prop}\label{gp}
Let $\Gamma \subset \mathrm{GL}_n(K)$ be a finitely generated linear group over a field $K$ of characteristic zero, and let $G$ be its Zariski-closure. Assume that the connected component $G^{\circ}$ is a nontrivial semi-simple group. Then for any $\gamma_1, \ldots , \gamma_r \in \Gamma$ there exists a semi-simple element $\gamma \in \Gamma \cap G^{\circ}$ whose eigenvalues are multiplicatively independent from those of $\gamma_1, \ldots , \gamma_r$ and for which the subgroup $\Lambda(\gamma)$ is nontrivial and torsion-free.
\end{prop}
\begin{proof}
Since $\Gamma$ is finitely generated, we can assume that the field $K$ is also finitely generated. By Selberg's Lemma (cf. \cite[6.11]{Rag}), we can choose a finite index subgroup $\Gamma' \subset \Gamma \cap G^{\circ}$ which is {\bf neat}, i.e. satisfies the property that for any $\delta \in \Gamma'$ the subgroup $\Lambda(\delta)$ is torsion-free; obviously, $\Gamma'$ is Zariski-dense in $G^{\circ}$.
Let $L$ be the extension of $K$ generated by the eigenvalues of $\gamma_1, \ldots , \gamma_r$, and let $\gamma \in \Gamma'$ be a regular semi-simple element provided by Lemma \ref{tezheng} for the group $G^{\circ}$. To show that this element is as required, we only need to verify condition (\ref{E:Y1}). However, any element  $x \in \Lambda(\gamma) \cap [\Lambda(\gamma_1) \cdots \Lambda(\gamma_r)]$ can be written in the form $\chi(\gamma)$ for some character $\chi$ of the maximal torus $Z_{G^{\circ}}(\gamma)^{\circ}$ and, on the other hand, belongs to $\Lambda(\gamma_1) \cdots \Lambda(\gamma_r) \subset L^{\times}$.
So, $x = 1$ by Lemma \ref{tezheng} completing the argument.
\end{proof}

\vskip.7mm

\noindent {\bf Remark 3.6.} 1. The assertion of Proposition \ref{gp} is {\it false} for virtually solvable subgroups $\Gamma \subset \mathrm{GL}_n(K)$.

\vskip.1mm

2. Using the adjoint representation, it is not difficult to show that any linear group $\Gamma \subset \mathrm{GL}_n(K)$ over a field of characteristic zero with semi-simple connected component of the Zariski-closure contains a finitely generated subgroup having the same Zariski-closure as $\Gamma$. This observation enables one to drop the assumption of finite generation of $\Gamma$ in all statements of this section.

\section{The key matrix statement}\label{S:Case1}

The main results of the paper will be derived in \S\ref{S:main} from the following more general statement that treats bounded generation by individual matrices without any reference to linear groups.
\begin{thm}\label{T:notsubset}
Assume that the matrices $\gamma_1, \ldots , \gamma_r \in \mathrm{GL}_n(\overline{\mathbb{Q}})$, with one possible exception, are semi-simple.
Then for any semi-simple matrix $\gamma \in \mathrm{GL}_n(\overline{\mathbb{Q}})$ that has an eigenvalue $\lambda \in \overline{\mathbb{Q}}^{\times}$ which is not a root of unity and for which $\langle \lambda \rangle \cap [\Lambda(\gamma_1) \ldots \Lambda(\gamma_r)] = \{ 1 \}$ (cf. Definition 3.4), the intersection $\langle \gamma \rangle \cap \langle \gamma_1 \rangle \cdots \langle \gamma_r \rangle$ is finite. In particular, $\langle \gamma \rangle \not\subset \langle \gamma_1 \rangle \cdots \langle \gamma_r \rangle$.
\end{thm}

\vskip1mm

\noindent {\bf 4A. Laurent's theorem and one application.} The proof of  Theorem \ref{T:notsubset} critically depends on the following result from Diophantine geometry -- see \cite[Theorem 2.7]{CorvajaZannier}.

\begin{thm}[Laurent's theorem]\label{dio} Let $\Omega$ be a finitely generated subgroup of $(\overline{\mathbb{Q}}^{\times})^N$ for some $N >0$, and let $\Sigma \subset \Omega$ be a subset. Then the Zariski closure of $\Sigma$ in the torus $T = (\mathbb{G}_m)^N$ is  a finite union of translates of algebraic subgroups of $T$.
\end{thm}

\vskip1mm

We will now establish the following important application of Laurent's theorem.

\begin{prop}\label{PP:1}
Suppose we are given a polynomial $f(x, x_1, \ldots , x_d) \in \overline{\mathbb{Q}}[x, x_1, \ldots , x_d]$ and algebraic numbers  $\mu, \mu_1, \ldots , \mu_d \in \overline{\mathbb{Q}}^{\times}$, with $\mu$ \underline{not} a root of unity. Assume that
$$
\langle \mu \rangle \cap \langle \mu_1, \ldots , \mu_d \rangle = \{ 1 \}.
$$
Then the set $M$ of integers $m \in \mathbb{Z}$ such that there exist $a_1(m), \ldots , a_d(m) \in \mathbb{Z}$ for which
$$
f\left(\mu^m, \mu_1^{a_1(m)}, \ldots , \mu_d^{a_d(m)}\right) = 0 \ \ \text{and} \ \ f\left(x, \mu^{a_1(m)}, \ldots , \mu_d^{a_d(m)}\right) \ \ \text{is non-constant},
$$
is finite.
\end{prop}
\begin{proof}
For each $m \in M$ we {\it fix} a $d$-tuple $(a_1(m), \ldots , a_d(m)) \in \mathbb{Z}^d$ as in the above description of $M$. Let $N = 1 + d$ and $T = (\mathbb{G}_m)^N$ be an $N$-dimensional $\mathbb{Q}$-split torus, the coordinate functions on which will be denoted $x, x_1, \ldots , x_d$. Furthermore, we let $\Omega$ denote the subgroup $\langle \mu \rangle \times \langle \mu_1 \rangle \times \cdots \times \langle \mu_d \rangle$ of $T(\overline{\mathbb{Q}}) = (\overline{\mathbb{Q}}^{\times})^N$; clearly, $\Omega$ is finitely generated. Consider the subset $\Sigma \subset \Omega$ consisting of
the elements $\sigma(m) := \left( \mu^m, \mu_1^{a_1(m)}, \ldots , \mu_d^{a_d(m)} \right)$ for  $m \in M$. By Laurent's Theorem \ref{dio}, the Zariski-closure $\overline{\Sigma}$ of $\Sigma$ in $T$ is of the form
\begin{equation}\label{EE:3}
\overline{\Sigma} = \bigcup_{\ell = 1}^b s_{\ell}T_{\ell}
\end{equation}
for some $s_{\ell} \in T(\overline{\mathbb{Q}})$ and some algebraic subgroups $T_{\ell}$ of $T$ for $\ell = 1, \ldots , b$.  We may assume that this decomposition of $\overline{\Sigma}$ is minimal, i.e. none of the cosets can be dropped.
\begin{lemma}\label{LL:1}
None of the subgroups $T_{\ell}$ is of the form $\mathbb{G}_m \times T'_{\ell}$ for some algebraic subgroup $T'_{\ell}$ of $(\mathbb{G}_m)^{N-1}$ (last
$(N-1)$ components).
\end{lemma}
\begin{proof}
By construction, the polynomial $f(x, x_1, \ldots , x_d)$ vanishes on $\Sigma$, hence also vanishes on $\overline{\Sigma}$. In particular, it vanishes on each coset $s_{\ell}T_{\ell}$. Suppose that for some $\ell \in \{ 1, \ldots , b \}$ we have $T_{\ell} = \mathbb{G}_m \times T'_{\ell}$. Then
\begin{equation}\label{EE:4}
s_{\ell} T_{\ell} = \mathbb{G}_m \times (s'_{\ell} T'_{\ell}),
\end{equation}
where $s'_{\ell}$ is the projection of $s_{\ell}$ to $(\mathbb{G}_m)^{N-1}$. The minimality of (\ref{EE:3}) implies that $s_{\ell}T_{\ell}$ contains an element  $\left(\mu^m, \mu_1^{a_1(m)}, \ldots , \mu_d^{a_d(m)}\right) \in \Sigma$. Due to (\ref{EE:4}), this means that $f$ vanishes on $\mathbb{G}_m \times \left\{\left(\mu_1^{a_1(m)}, \ldots , \mu_d^{a_d(m)}\right)\right\}$. But this is impossible since by our assumption the polynomial $f\left(x, \mu_1^{a_1(m)}, \ldots , \mu_d^{a_d(m)}\right)$ is nonconstant.   \end{proof}

Now, if $M$ is infinite, then by the pigeonhole principle, one can find $m_1 , m_2 \in M$, $m_1 \neq m_2$, such that the elements $\sigma(m_1)$ and $\sigma(m_2)$ belong to the {\it same} coset $s_{\ell} T_{\ell}$. Then
\begin{equation}\label{EE:1}
\sigma(m_2)\sigma(m_1)^{-1} = \left( \mu^{m_2 - m_1}, \mu_1^{a_1(m_2) - a_1(m_1)}, \ldots , \mu_d^{a_d(m_1) - a_d(m_2)}   \right)
\end{equation}
belongs to $T_{\ell}$. According to \cite[Ch. III, 8.2]{borel}, the subgroup $T_{\ell}$ is the intersection of the kernels of all characters of $T$ that vanish on it. Writing a character $\chi \in X(T)$ as
\begin{equation}\label{EE:2}
\chi(x, x_1, \ldots , x_d) = x^k x_1^{k_1} \cdots x_d^{k_d},
\end{equation}
we observe that if $k = 0$ then $\ker \chi$ is of the form $\mathbb{G}_m \times T'$. So, it follows from Lemma \ref{LL:1} that there exists a character $\chi \in X(T)$ that vanishes on $T_{\ell}$ and for which in the corresponding presentation (\ref{EE:2}) we have $k \neq 0$. Furthermore, since $\sigma(m_2)\sigma(m_1)^{-1} \in T_{\ell}$, it follows from (\ref{EE:1}) that
$$
\mu^{k(m_2 - m_1)} \in \langle \mu_1, \ldots , \mu_d \rangle.
$$
We obtain a contradiction completing thereby the proof of Proposition \ref{PP:1}.
\end{proof}

\vskip1mm

\noindent {\bf 4B. The two cases in the proof of Theorem \ref{T:notsubset}.} To prove the theorem, it is enough to consider the following two cases.

\vskip2mm

\noindent {\bf Case 1.} {\it All elements $\gamma_1, \ldots , \gamma_r$ are semi-simple.}

\vskip1mm

\noindent {\bf Case 2.} {\it For some $s \in \{1, \ldots , r\}$, the elements $\gamma_1, \ldots , \gamma_{s-1}, \gamma_{s+1}, \ldots \gamma_r$ are semi-simple and the element $\gamma_s \neq 1$ is unipotent.}

\vskip2mm

Indeed, in view of the assumptions made in the statement of Theorem \ref{T:notsubset}, the only situation not covered by Case 1 is  where among the elements $\gamma_1, \ldots , \gamma_r$  exactly one element, say $\gamma_s$ $(1 \leq s \leq r)$, is not semi-simple. The Jordan decomposition $\gamma_s = \sigma \upsilon$, where $\sigma$ and $\upsilon$  are commuting matrices with $\sigma$ semi-simple and $\upsilon$ unipotent, leads to the inclusion
$$
\langle \gamma_1 \rangle \cdots \langle \gamma_r \rangle \ \subset \ \langle \gamma_1 \rangle \cdots \langle \gamma_{s-1} \rangle \langle \sigma \rangle \langle \upsilon \rangle \langle \gamma_{s+1} \rangle \cdots \langle \gamma_r \rangle.
$$
Thus, it is enough to prove that under the assumptions made in Theorem \ref{T:notsubset}, the intersection $$\langle \gamma \rangle \cap \left[ \langle \gamma_1 \rangle \cdots \langle \gamma_{s-1} \rangle \langle \sigma \rangle \langle \upsilon \rangle \langle \gamma_{s+1} \rangle \cdots \langle \gamma_r \rangle\right]$$ is finite. On the other hand, we have
$\Lambda(\sigma) = \Lambda(\gamma_s)$ and $\Lambda(\upsilon) = \{ 1 \}$, so the assumption on $\lambda$ in the theorem translates into
$$
\langle \lambda \rangle \cap \left[ \Lambda(\gamma_1) \cdots \Lambda(\gamma_{s-1}) \Lambda(\sigma) \Lambda(\upsilon) \Lambda(\gamma_{s+1}) \cdots \Lambda(\gamma_r)\right] = \{ 1 \}.
$$
Clearly, after re-labeling the elements $\gamma_1, \ldots , \gamma_{s-1}, \sigma, \upsilon, \gamma_{s+1}, \ldots , \gamma_r$  precisely fit the situation
considered in Case 2. Thus, the combination of Cases 1 and 2 covers all possibilities in Theorem \ref{T:notsubset}.

\vskip2mm

\noindent {\bf 4C. Proof of Theorem \ref{T:notsubset} in Case 1.} As usual, we let $\mathrm{diag}(u_1, \ldots , u_n)$ denote the diagonal matrix with  diagonal entries $u_1, \ldots , u_n$. Since $\gamma, \gamma_1, \ldots , \gamma_r$ are semi-simple, there exist $g, g_1, \ldots , g_r \in \mathrm{GL}_n(\overline{\mathbb{Q}})$ such that
$$
g^{-1}\gamma g = \mathrm{diag}(\lambda_1, \ldots , \lambda_n) \ \ \text{and} \ \ g_i^{-1} \gamma_i g_i = \mathrm{diag}(\lambda_{i1}, \ldots , \lambda_{in}) \ \ \text{for} \ \ i = 1, \ldots , r.
$$
Besides, we may assume that $\lambda = \lambda_1$. We introduce $d = rn$ indeterminates $x_{11}, \ldots , x_{1n}, \ldots , x_{r1}, \ldots , x_{rn}$, and let $p(x_{11}, \ldots , x_{rn})  \in \overline{\mathbb{Q}}[x_{11}, \ldots , x_{rn}]$ denote the polynomial representing
the $(1,1)$-entry of the matrix
$$
g^{-1} \cdot \left[  \prod_{i = 1}^r \left(g_i \cdot \mathrm{diag}(x_{i1}, \ldots , x_{in}) \cdot g_i^{-1}\right)  \right] \cdot g.
$$
We set $f(x, x_{11}, \ldots , x_{rn}) = x - p(x_{11}, \ldots , x_{rn})$ observing that the polynomial $f(x, x_{11}^0, \ldots , x_{rn}^0)$ is non-constant for any $x_{11}^0, \ldots , x_{rn}^0 \in \overline{\mathbb{Q}}$.

Now, let $J = \{\, m \in \mathbb{Z} \ \vert \ \gamma^m \in \langle \gamma_1 \rangle \cdots \langle \gamma_r \rangle \, \}$. Then for each $m \in J$ we can make a choice of integers $a_1(m), \ldots , a_r(m)$ so that
$$
\gamma^m = \gamma_1^{a_1(m)} \cdots \gamma_r^{a_r(m)},
$$
and consequently
$$
\lambda^m = p\left(\lambda_{11}^{a_1(m)}, \ldots , \lambda_{1n}^{a_1(m)}, \ldots , \lambda_{r1}^{a_r(m)}, \ldots , \lambda_{rn}^{a_r(m)}\right).
$$
Then
$$
f\left(\lambda^m, \lambda_{11}^{a_1(m)}, \ldots , \lambda_{rm}^{a_r(m)}\right) = 0.
$$
This means that $J$ is contained in the set $M$ constructed in Proposition \ref{PP:1} for the polynomial $f$ and taking $\mu$ to be $\lambda$ and $\mu_1, \ldots , \mu_d$ to be $\lambda_{11}, \ldots , \lambda_{rn}$. We note that then
$$
\langle \mu_1, \ldots , \mu_d \rangle = \Lambda(\gamma_1) \cdots \Lambda(\gamma_r).
$$
It follows that the assumptions of Proposition \ref{PP:1} do hold in our situation, so $M$ is finite. Therefore, $J$ is also finite, and our claim follows.

\vskip2mm

\noindent {\bf 4D. Proof of Theorem \ref{T:notsubset} in Case 2.} We now assume that the elements $\gamma_1, \ldots , \gamma_{s-1}, \gamma_{s+1}, \ldots , \gamma_r$ are semi-simple and the element $\gamma_s \neq 1$ is unipotent. Again, we need to show that the set $$J = \left\{ m \in \mathbb{Z} \ \vert \ \gamma^m \in \langle \gamma_1 \rangle \cdots \langle \gamma_r \rangle \right\}$$ is finite. Emulating the construction in 4C, we find $g, g_1, \ldots , g_{s-1}, g_{s+1}, \ldots , g_r \in \mathrm{GL}_n(\overline{\mathbb{Q}})$ so that
$$
g^{-1} \gamma g = \mathrm{diag}(\lambda_1, \ldots , \lambda_n) \ \ \text{and} \ \ g_i^{-1} \gamma_i g_i = \mathrm{diag}(\lambda_{i1}, \ldots ,
\lambda_{in}) \ \ \text{for} \ \ i = 1, \ldots , s-1, s+1, \ldots , r,
$$
where we may assume that $\lambda = \lambda_1$. On the other hand, $\gamma_s = I_n + \nu$ where $\nu \neq 0$ and $\nu^n = 0$. Then by the binomial expansion, for any $m \in \mathbb{Z}$ we have
$
\displaystyle \gamma_s^m = \sum_{k = 0}^{n-1} {m \choose k} \nu^k
$
where as usual
$$
{m \choose k}  = \frac{m(m-1) \cdots (m - k + 1)}{k!} \ \ \text{for} \ \ k \geq 1 \ \ \text{and} \ \ {m \choose 0} = 1 \ \ \text{for any} \ \ m \in \mathbb{Z}.  $$
It follows that there exists an $(n\times n)$-matrix $A(z)$ with entries in $\overline{\mathbb{Q}}[z]$ such that  $\gamma_s^m = A(m)$ for all $m \in \mathbb{Z}$. We now introduce the indeterminates $x_{ij}$ for $i \in \{1, \ldots , s-1, s+1, \ldots , r\}$ and $j \in \{1, \ldots , n\}$. Then for any $\alpha , \beta \in \{1, \ldots , n\}$ we let $p_{\alpha\beta}(x_{ij} , z)$ denote the polynomial in $\overline{\mathbb{Q}}[x_{11}, \ldots , x_{rn}, z]$ that represent the $(\alpha , \beta)$-entry of the matrix
$$
P(x_{ij} , z) := g^{-1} \cdot \left[ \prod_{i = 1}^{s-1}\left(g_i \cdot \mathrm{diag}(x_{i1}, \ldots , x_{in}) \cdot g_i^{-1}\right) \cdot A(z) \cdot
 \prod_{i = s+1}^{r}\left(g_i \cdot \mathrm{diag}(x_{i1}, \ldots , x_{in}) \cdot g_i^{-1}\right)     \right] g.
$$
Set $q(x, x_{ij}, z) := x - p_{11}(x_{ij}, z)$. For each $m \in J$ we can fix $a_1(m), \ldots , a_r(m) \in \mathbb{Z}$ such that
%\begin{equation}\label{EEE:1}
$$
\gamma^m = \gamma_1^{a_1(m)} \cdots \gamma_r^{a_r(m)}.
$$
%\end{equation}
Then for all $m \in J$ we have
\begin{equation}\label{EEE:2}
q\left(\lambda^m, \lambda_{11}^{a_1(m)}, \ldots , \lambda_{rn}^{a_r(m)}, a_s(m)\right) = 0
\end{equation}
Furthermore, since the matrix $g^{-1} \gamma g$ is diagonal, we in addition have
\begin{equation}\label{EEE:3}
p_{\alpha\beta}\left(\lambda_{11}^{a_1(m)}, \ldots , \lambda_{rn}^{a_r(m)}, a_s(m)\right) = 0 \ \ \text{whenever} \ \ \alpha \neq \beta,
\end{equation}
for all $m \in J$. Assume that $J$ is infinite. To obtain a contradiction, we will eliminate $z$ from the pair of polynomials $q(x, x_{ij}, z)$ and $\tilde{p}_{\alpha\beta}(x_{ij}, z)$ for {\it some} $\alpha \neq  \beta$, where $\tilde{p}_{\alpha\beta}$ is a suitable truncation of $p_{\alpha\beta}$, to generate a polynomial $f(x, x_{ij})$ having the following property: the set $M$ constructed  in Proposition \ref{PP:1} for this $f$ by taking $\mu = \lambda$ and $\mu_1, \ldots , \mu_d$ to be $\lambda_{11}, \ldots , \lambda_{rn}$ contains an infinite subset $J' \subset J$. It will be obvious that the assumptions of the proposition hold in our situation which will imply  $M$ is actually finite and provide the required contradiction.

We will now pick a suitable pair of indices $\alpha , \beta \in \{1, \ldots , n\}$, $\alpha \neq \beta$, and construct a required truncation $\tilde{p}_{\alpha\beta}(x_{ij}, z)$. For any such pair we write
$$
p_{\alpha\beta}(x_{ij}, z) = \psi^{\alpha\beta}_{t_{\alpha\beta}}(x_{ij}) z^{t_{\alpha\beta}} + \cdots + \psi^{\alpha\beta}_0(x_{ij}) \ \ \text{with} \ \
\psi^{\alpha\beta}_k(x_{ij}) \in \overline{\mathbb{Q}}[x_{ij}],
$$
and then for $k = 1, \ldots , t_{\alpha\beta}$ set
$$
\Psi^{\alpha\beta}_k = \left\{ m \in J \ \vert \  \psi^{\alpha\beta}_k\left(\lambda_{11}^{a_1(m)}, \ldots , \lambda_{rn}^{a_r(m)}\right) = 0 \right\}.
$$
\begin{lemma}
$\displaystyle \bigcap_{\alpha \neq \beta} \left( \bigcap_{k = 1}^{t_{\alpha\beta}} \Psi^{\alpha\beta}_k   \right) = \varnothing$.
\end{lemma}
\begin{proof}
Assume the contrary, and let $m$ be an element of this intersection. Then all off-diagonal entries of the matrix $\overline{P}(z) := P\left(\lambda_{11}^{a_1(m)}, \ldots , \lambda_{rn}^{a_r(m)}, z \right)$ are independent of $z$. Thus, since $$\overline{P}(a_s(m)) = g^{-1} \cdot  \left(\gamma_1^{a_1(m)} \cdots \gamma_s^{a_s(m)} \cdots  \gamma_r^{a_r(m)}\right) \cdot g = g^{-1} \gamma^m g $$ is diagonal, the matrix $\overline{P}(z)$ is diagonal for any $z$. In particular,
$$
\overline{P}(a_s(m) + 1) = g^{-1} \cdot \left(\gamma_1^{a_1(m)} \cdots \gamma_s^{a_s(m) + 1} \cdots \gamma_r^{a_r(m)}\right) \cdot g
$$
is diagonal. Then
$$
\overline{P}(a_s(m))^{-1} \overline{P}(a_s(m) + 1) = \left(\gamma_{s+1}^{a_{s+1}(m)} \cdots \gamma_r^{a_r(m)} g\right)^{-1} \cdot \gamma_s \cdot \left(\gamma_{s+1}^{a_{s+1}(m)} \cdots \gamma_r^{a_r(m)} g\right)
$$
is also diagonal, contradicting the fact that $\gamma_s$ is a {\it nontrivial} unipotent matrix.
\end{proof}

Thus,
$$
J = \bigcup_{\alpha \neq \beta} \left( \bigcup_{k = 1}^{t_{\alpha\beta}} J \setminus \Psi^{\alpha\beta}_k \right),
$$
so there exist $\alpha \neq \beta$ and $k \in \{1, \ldots , t_{\alpha\beta}\}$ such that $J \setminus \Psi^{\alpha\beta}_k$ is infinite. We fix one such pair $(\alpha , \beta)$ and let $t \geq 1$ be the {\it largest} integer $\leq t_{\alpha\beta}$ for which $J \setminus \Psi^{\alpha\beta}_t$ is infinite.
Since $J \setminus \Psi^{\alpha\beta}_k$ is finite for $t < k \leq t_{\alpha\beta}$, the set
$$
J' := \left(J \setminus \Psi^{\alpha\beta}_t\right)  \bigcap \left( \bigcap_{k = t+1}^{t_{\alpha\beta}} \Psi^{\alpha\beta}_k  \right)
$$
is still infinite. Let
$$
\tilde{p}_{\alpha\beta}(x_{ij}, z) = \psi^{\alpha\beta}_{t}(x_{ij}) z^{t} + \cdots + \psi^{\alpha\beta}_0(x_{ij}).
$$
It follows from (\ref{EEE:3}) and our construction that for all $m \in J'$ we have
\begin{equation}\label{EEE:10}
\tilde{p}_{\alpha\beta}\left(\lambda_{11}^{a_1(m)}, \ldots , \lambda_{rn}^{a_r(m)}, a_s(m)\right) = 0 \ \ \text{and} \ \
\psi^{\alpha\beta}_{t}\left(\lambda_{11}^{a_1(m)}, \ldots , \lambda_{rn}^{a_r(m)}\right) \neq 0.
\end{equation}

Let $p_{11}(x_{ij}, z) = \phi_e(x_{ij})z^e + \cdots + \phi_0(x_{ij})$ (this polynomial was denoted $p$ in the proof of Case 1) so that
$$
q(x, x_{ij}, z) = -\phi_e(x_{ij})z^e - \cdots - \phi_1(x_{ij})z + (x - \phi_0(x_{ij})).
$$
Referring to \cite[Ch. IV, \S8]{Langalgebra} for the basic facts about resultants, we consider the resultant of the polynomials $q$ and $\tilde{p}_{\alpha\beta}$ with respect to the variable $z$:
\[
R_z(q,\tilde{p}_{\alpha \beta})= \begin{vmatrix}
-\phi_e(x_{ij}) &\dots & -\phi_1(x_{ij}) & x-\phi_0(x_{ij})&&&\\[16pt]
& \ddots&\ddots& \ddots&\ddots&& \\[16pt]
&&-\phi_e(x_{ij}) &\dots &-\phi_1(x_{ij})&x-\phi_0(x_{ij})\\[16pt]
\psi_t^{\alpha\beta}(x_{ij}) &\psi_{t-1}^{\alpha\beta}(x_{ij})&\dots &\psi_0^{\alpha\beta}(x_{ij}) && & \\[16pt]
&\ddots&\ddots&\ddots& \ddots&& \\[16pt]
&& \psi_t^{\alpha\beta}(x_{ij})&\psi_{t-1}^{\alpha\beta}(x_{ij})&\dots & \psi_0^{\alpha\beta}(x_{ij})&
\end{vmatrix}
\hspace{0.5em}
\begin{tabular}{l}
$\left.\lefteqn{\phantom{\begin{matrix} x-\phi_0(x_{ij})\\ \\  \\ \\ \ddots\\ x-\phi_0(x_{ij})\ \end{matrix}}}\right\}t\text{ rows}$\\
$\left.\lefteqn{\phantom{\begin{matrix} \psi_0^{\alpha\beta}(x_{ij})\\ \\ \\ \\  \ddots\\ \psi_0^{\alpha\beta}(x_{ij})\ \end{matrix}}} \right\}e\text{ rows}$
\end{tabular}
\]
We will view $R_z(q , \tilde{p}_{\alpha\beta})$ as a polynomial $f(x, x_{ij}) \in \overline{\mathbb{Q}}[x, x_{ij}]$. It is easy to see that $\deg_x f = t$ and the coefficient of $x^t$ is $\pm [\psi^{\alpha\beta}_t(x_{ij})]^e$. It follows from (\ref{EEE:2}) and (\ref{EEE:10}) that for any $m \in J'$ the polynomials
$$
q\left(\lambda^m, \lambda_{11}^{a_1(m)}, \ldots , \lambda_{rn}^{a_r(m)}, z\right) \ \ \text{and} \ \ \tilde{p}_{\alpha\beta}\left(\lambda_{11}^{a_1(m)}, \ldots , \lambda_{rn}^{a_r(m)}, z\right)
$$
have a common root $z = a_s(m)$, implying that
$$
f\left(\lambda^m, \lambda_{11}^{a_1(m)}, \ldots , \lambda_{rn}^{a_r(m)}\right) = 0
$$
(cf. \cite[Ch. IV, Prop. 8.1]{Langalgebra}).
It also follows from (\ref{EEE:10}) that $f\left(x, \lambda_{11}^{a_1(m)}, \ldots , \lambda_{rn}^{a_r(m)}  \right)$ is a non-constant polynomial. This means that $J'$ is contained in the set $M$ defined in Proposition \ref{PP:1} for our polynomial $f$ and by taking $\mu = \lambda$ and $\mu_1, \ldots , \mu_d$ to be $\lambda_{11}, \ldots , \lambda_{rn}$. Since $\langle \lambda \rangle \cap \left[\Lambda(\gamma_1) \cdots \Lambda(\gamma_r)\right] = \{ 1 \}$ by assumption, the condition $\langle \mu \rangle \cap \langle \mu_1, \ldots , \mu_d \rangle = \{ 1 \}$ of Proposition \ref{PP:1} holds, allowing us to conclude that the set $M$ is finite. This contradicts the fact that $J'$ is infinite and completes the proof of Theorem~\ref{T:notsubset}. \hfill $\Box$

%\vskip5mm

\section{Proof of the main results}\label{S:main}

\noindent {\it Proof of Theorem \ref{mainthm}.} We need to show that if a linear group $\Gamma \subset \mathrm{GL}_n(K)$ admits a presentation (\ref{E:BG}) where at most one of the $\gamma_i$'s fails to be semi-simple, then $\Gamma$ is virtually solvable. According to Reduction 2.3, it is enough to consider the case where $K$ is a number field. Letting $G$ denote the Zariski-closure of $\Gamma$, we consider the radical $R$ of
the connected component $G^{\circ}$ (which is a normal subgroup of $G$) and the corresponding canonical morphism $\varphi \colon G \to G/R = G'$. Clearly, $G$ and $R$, hence also $G'$ and $\varphi$, are defined over $K$, and in particular, we can choose a faithful $K$-defined representation $G' \hookrightarrow \mathrm{GL}_{n'}$. Since $\varphi$ takes semi-simple elements to semi-simple elements (cf. \cite[Ch. I, 4.4]{borel}), among the elements $\gamma'_i = \varphi(\gamma_i)$ there is at most one non-semi-simple. Furthermore, the subgroup $\Gamma' = \varphi(\Gamma)$ of $G'(K) \subset \mathrm{GL}_{n'}(K)$ has a presentation
$$
\Gamma' = \langle \gamma'_1 \rangle \cdots \langle \gamma'_r \rangle.
$$
If we assume that $\Gamma$ is not virtually solvable, then the connected component of the Zariski-closure $G'$ of $\Gamma'$ will be a {\it nontrivial} semi-simple group. Thus, what we need to show is that if a linear group $\Gamma \subset \mathrm{GL}_n(K)$ over a number field $K$ is such that the connected component $G^{\circ}$ of its Zariski-closure $G$ is a nontrivial semi-simple group, then $\Gamma$ cannot have a presentation (\ref{E:BG}) where all $\gamma_i$, with one possible exception, are semi-simple.

Assume the contrary. Using Proposition \ref{gp}, we can find a semi-simple element $\gamma \in \Gamma \cap G^{\circ}$ whose eigenvalues are multiplicatively independent from those of the elements $\gamma_1, \ldots , \gamma_r$ in (\ref{E:BG}) and for which the subgroup $\Lambda(\gamma)$ is nontrivial and torsion-free. In particular, $\gamma$ has an eigenvalue $\lambda$ which is not a root of unity and for which $\langle \lambda \rangle \cap \left[ \Lambda(\gamma_1) \cdots \Lambda(\gamma_r) \right] = \{ 1 \}$. Then according to Theorem \ref{T:notsubset} we have $\langle \gamma \rangle \not\subset \langle \gamma_1 \rangle \cdots \langle \gamma_r \rangle$, which obviously contradicts (\ref{E:BG}). \hfill $\Box$

\vskip2mm

We would like to point out that as proved in \cite{Vsem}, every matrix in $\Gamma_p = \mathrm{SL}_2(\mathbb{Z}[1/p])$ (where $p$ is a prime) is a product of  $\leq 5$ elementaries, while the example given in \cite[\S 5]{MRS} demonstrates that for $p > 7$ there are matrices in $\Gamma_p$ that are not products of 4 elementaries. These facts suggest that it may be possible to upgrade Theorem \ref{mainthm} to a statement that for a non-virtually solvable linear group $\Gamma$ over a field of characteristic zero, any presentation (\ref{E:BG}) must involve at least 5 non-semi-simple elements, with 5 being the best possible bound.

\vskip2mm

\noindent {\it Proof of Corollary \ref{C:1}.} Let $\Gamma \subset \mathrm{GL}_n(K)$ be an anisotropic linear group over a field $K$ of characteristic zero that has bounded generation. Then $\Gamma$ is boundedly generated by semi-simple elements, hence virtually solvable by Theorem \ref{mainthm}. So, if we let $G$ denote the Zariski-closure of $\Gamma$, then the connected component $G^{\circ}$ is solvable. Let $U$ be the unipotent radical of $G^{\circ}$; then the quotient $G^{\circ}/U$ is a torus (cf. \cite[Ch. III, Theorem 10.6]{borel}). On the other hand, since $\Gamma$ is anisotropic, the restriction  of the quotient map $G^{\circ} \to G^{\circ}/U$ to $\Gamma \cap G^{\circ}$ is injective, so the latter is isomorphic to a subgroup of the abelian group $(G^{\circ}/U)(K)$. Thus, $\Gamma \cap G^{\circ}$ is abelian, making $\Gamma$ virtually abelian. The finite generation of $\Gamma$ is obvious. Conversely, if $\Gamma$ is finitely generated and has an abelian subgroup $\Delta$ of finite index, then $\Delta$ itself is finitely generated, hence has bounded generation, implying that $\Gamma$ has bounded generation as well. \hfill $\Box$

\vskip2mm

\noindent {\it Proof of Theorem \ref{zdlx}.} It is well-known and easy to show that a finite index subgroup of a group has bounded generation if and only if the group does. (Incidentally, this implies that if one $S$-arithmetic subgroup of an algebraic group has bounded generation then all $S$-arithmetic subgroups do.) So, passing to the connected component, we may assume $G$ to be connected. Let $\varphi \colon G \to D$ be the quotient map to the semi-simple group $D = G/R$ (everything is defined over $K$). Then for any semi-simple subgroup $\mathscr{H}$ of $G$, the restriction $\varphi \vert \mathscr{H}$ is an isogeny. This means that if $\mathscr{H}$ is a $K$-defined semi-simple $K$-anisotropic subgroup of $G$ with $\mathscr{H}_S$ noncompact then $H = \varphi(\mathscr{H})$ is a subgroup of $D$ with similar properties. At the same time, if $\Delta \subset G(K)$ is an $S$-arithmetic subgroup  having bounded generation then $\varphi(\Delta)$ is an $S$-arithmetic subgroup of $D(K)$ (cf. \cite[Theorem 5.9]{PR}) that also has bounded generation. Thus, replacing $G$ by $D$, we may assume that $G$ is semi-simple and $H$ is a normal subgroup of $G$ having the properties specified in the statement of the theorem. Then there is a surjective $K$-defined morphism $\psi \colon G \to \overline{H}$ where $\overline{H}$ is the adjoint group of $H$. Furthermore, there is a surjective $K$-defined morphism $\eta \colon \overline{H} \to H'$ to a $K$-simple group $H'$ with noncompact $H'_S$. Then, again, given an $S$-arithmetic subgroup $\Delta$ of $G(K)$ having bounded generation, the image $\Delta' = (\eta \circ \psi)(\Delta)$ is an $S$-arithmetic subgroup of $H'(K)$ also having bounded generation. Since $H$ is $K$-anisotropic, so is $H'$, implying that $\Delta'$ is an anisotropic linear group. Thus, according to Corollary \ref{C:1}, the group $\Delta'$ must be virtually abelian. On the other hand, by Borel's Density Theorem for $S$-arithmetic groups (cf. \cite[Theorem 4.10]{PR} for usual arithmetic subgroups), $\Delta'$ is Zariski-dense in $H'$. However, being a connected group that coincides with its derived subgroup, $H'$ cannot have a Zariski-dense virtually abelian subgroup. A contradiction, proving the theorem. \hfill $\Box$

\vskip1mm

To prove Corollary \ref{negative}, we need to recall that there are series of $S$-arithmetic subgroups of {\it anisotropic} absolutely almost simple simply connected algebraic groups over number fields, where $S$ is a finite set of valuations of the base field, for which the congruence kernel is known to be finite (cf. \cite[Theorems 3 and 4]{Rap}). It is well-known that these groups are finitely generated (cf. \cite[Theorem 5.11]{PR}), and also residually finite. Furthermore, the profinite completions of these groups have property $(\mathrm{BG})_{\mathrm{pr}}$ of bounded generation as  profinite groups \cite[Theorem 2]{PR-CSP}. At the same time, according to Theorem \ref{zdlx} the groups themselves do not have property $(\mathrm{BG})$ of bounded generation as discrete groups. \hfill $\Box$

\vskip1mm

\section{Final remarks}\label{zhujie}

First, here is an example of a solvable finitely generated linear group without bounded generation which shows that Theorem \ref{mainthm} is not a criterion.

\vskip1mm

\addtocounter{thm}{1}

\noindent {\bf Example \ref{zhujie}.1} \ Let $x$ be a variable. Consider the group
$$
\Gamma = \left\{ \ \left. \left( \begin{array}{cc} x^i & a \\ 0 & x^{-i} \end{array} \right) \ \ \right\vert \ \ i \in \mathbb{Z}, \ \ a \in A := \mathbb{Z}[x , x^{-1}] \ \right\}.
$$
It is solvable and finitely generated. In fact, it is generated by the following three matrices
$$
 t = \left( \begin{array}{cc} x & 0 \\ 0 & x^{-1} \end{array} \right), \ \ u = \left( \begin{array}{cc} 1 & 1 \\ 0 & 1 \end{array}  \right), \ \ \text{and} \ \ v = \left( \begin{array}{cc} 1 & x \\ 0 & 1 \end{array} \right).
$$
To see that these matrices indeed generate $\Gamma$, one observes that $\Gamma = TU$ (semi-direct product) where $T = \langle t \rangle$ and $U = \left\{ \left. \left( \begin{array}{cc} 1 & a \\ 0 & 1 \end{array} \right) \ \right\vert \ a \in A \right\}$. Then our claim follows from the relations
$$
t^i u t^{-i} = \left( \begin{array}{cc} 1 & x^{2i} \\ 0 & 1 \end{array}  \right) \ \ \text{and} \ \ t^i v t^{-i} = \left( \begin{array}{cc} 1 & x^{2i+1} \\ 0 & 1 \end{array}   \right).
$$
In the preliminary versions of this paper (see arXiv:2101.09386, v. 1 and 2), the fact that the group $\Gamma$, which strongly resembles the lampligher groups, does not have bounded generation was verified by a direct computation. Subsequently, D.~Segal, B.~Sury and both anonymous referees suggested to replace these computations with references to some general results, which we have gladly implemented. First, Corollary 1.5 in Pyber-Segal \cite{PySe} states that a finitely generated residually finite virtually solvable group has bounded generation if and only if it has finite rank\footnote{This means that there exists an integer $r > 0$ such that every finitely generated subgroup can be generated by $\leq r$ elements}. Since $\Gamma$ is a finitely generated linear group, it is automatically residually finite, and is  also solvable. On the other hand, the subgroup $U$ above obviously has infinite rank. Thus, $\Gamma$ also has infinite rank, hence is not boundedly generated. (We note that since $\Gamma$ is linear, instead of {\it loc. cit.} it is enough to quote Corollary 2 in \cite{Segal2}.) Second, Theorem 1.1 of Nikolov and Sury \cite{NikSury} states that the wreath product $A \wr B$ of nontrivial groups $A$ and $B$ has bounded generation if and only if $A$ has bounded generation and $B$ is finite. It is easy to see that $\Gamma$ can be written as the wreath product $(\mathbb{Z}
\oplus \mathbb{Z}) \wr \mathbb{Z}$, hence does not have (BG) according to this criterion.

\vskip1mm

For comparison, we recall that a solvable group of {\it integral matrices} is polycyclic (see \cite[p. 26]{Segal}), hence has bounded generation. Furthermore,
every virtually solvable {\it anisotropic} linear group is virtually abelian, and therefore automatically has bounded generation in case it is finitely generated.

\vskip1mm

One of the referees of the present paper suggested to consider finitely generated anisotropic linear groups that can be factored as a product of finitely many abelian (rather than cyclic) subgroups and asked if one can prove the assertion of Corollary \ref{C:1} in this more general situation. In order to provide the affirmative answer to this question, we first prove the following statement of independent interest.

\begin{prop}\label{P:abelian-sub}
Let $\Gamma \subset \mathrm{GL}_n(K)$ be a finitely generated linear group. Then any commutative \emph{anisotropic} subgroup $\Delta \subset \Gamma$ is finitely generated.
\end{prop}
\begin{proof}
Since $\Gamma$ is finitely generated, one can find a finitely generated subring $R \subset K$ such that $\Gamma \subset \mathrm{GL}_n(R)$. Let $A$ be the $K$-subalgebra of the matrix algebra $\mathrm{M}_n(K)$ generated by $\Delta$. Since $\Delta$ is commutative and consists of semi-simple elements, $A$ is conjugate over an algebraic closure of $K$ to a subalgebra of the algebra of diagonal matrices, and in particular is \emph{reduced} (i.e., does not contain any  nonzero nilpotent elements). Now, fix a basis $\omega_1, \ldots , \omega_r$ of $A$ over $K$, and consider the corresponding structure constants $a_{ij}^k$ defined by the equations
$$
\omega_i \omega_j = \sum_{k = 1}^r a_{ij}^k \omega_k, \ \ i,j = 1, \ldots , r.
$$
Then for \emph{any} finitely generated subring $R' \subset K$ containing all the $a_{ij}^k$'s, the sum
$$
\mathscr{A}(R') := \sum_{k = 1}^r R' \omega_k
$$
is a finitely generated subring of $A$. Since $A$ is reduced, it follows from \cite{Samuel} that the group of units $\mathscr{A}(R')^{\times}$ is finitely generated.

For an element $a \in A$, we let $a_1, \ldots , a_r$ denote its coordinates with respect to the basis $\omega_1, \ldots , \omega_r$ of $A$ over $K$, and $a'_{11}, \ldots a'_{nn}$ its coordinates with respect to the standard basis $e_{11}, \ldots , e_{nn}$ of $\mathrm{M}_n(K)$. Then there exist linear functions $\lambda_1(x_{11}, \ldots , x_{nn}), \ldots , \lambda_r(x_{11}, \ldots , x_{nn})$ over $K$ such that $$a_i = \lambda_i(a'_{11}, \ldots , a'_{nn}) \ \  \text{for all} \ \ a \in A.$$
Let $R'$ be the (finitely generated) subring of $K$ obtained by adjoining to $R$ the structure constants $a_{ij}^k$ and the coefficients of the linear functions $\lambda_1, \ldots , \lambda_r$. Then it follows from our construction that the fact that $\Delta \subset \mathrm{GL}_n(R)$ yields the inclusion $\Delta \subset \mathscr{A}(R')^{\times}$ in the above notations. However, as we mentioned above, $\mathscr{A}(R')^{\times}$ is a finitely generated abelian group, and the finite generation of $\Delta$ follows.
\end{proof}

We note that the above argument proves a slightly more general fact: Let $K$ be a field. Then for any finitely generated subring $R \subset K$, a commutative subgroup $\Delta \subset \mathrm{GL}_n(R)$ consisting of semi-simple elements is finitely generated. In particular, for any $K$-torus $T \subset \mathrm{GL}_n$, the group of $R$-points $T(R) := T \cap \mathrm{GL}_n(R)$ is finitely generated. (It is worth pointing out that these results are valid in any characteristic.) 

Now, let $\Gamma \subset \mathrm{GL}_n(K)$ be a finitely generated anisotropic linear group over a field $K$ of characteristic zero. Then it follows from Proposition \ref{P:abelian-sub} that every abelian subgroup $\Delta \subset \Gamma$ is finitely generated, hence has bounded generation. Thus, if $\Gamma$ is a product of finitely many abelian subgroups then it actually has bounded generation. So, invoking Corollary \ref{C:1} we obtain the following result that confirms  the referee's expectation.

\begin{cor}
Let $\Gamma \subset \mathrm{GL}_n(K)$ be a finitely generated anisotropic linear group over a field $K$ of characteristic zero. If $\Gamma$ is a product of finitely many abelian subgroups then $\Gamma$ is virtually abelian.
\end{cor}

It is known that every linear group over a field of positive characteristic that has bounded generation is virtually abelian \cite[Theorem 1]{Abert}. So, using Proposition \ref{P:abelian-sub}, we conclude that if an anisotropic linear group over a field of positive characteristic is a product of finitely many abelian subgroups then just as in the case of characteristic zero it must be virtually abelian.

\vskip1mm

Next, the lack of bounded generation in non-virtually abelian anisotropic linear groups casts new light on the problem of bounded generation of free products
with amalgamation. Let $\Gamma = \Gamma_1 *_{\Gamma_0} \Gamma_2$ be a free amalgamated product of two groups $\Gamma_1$ and $\Gamma_2$ along a common subgroup $\Gamma_0$. It was shown in \cite{F2} and \cite{Gr2} that if the number of double cosets $\Gamma_0 \backslash \Gamma_i / \Gamma_0$ is $> 2$ for at least one $i \in \{ 1 , 2\}$ then $\Gamma$ does not have bounded generation. On the other hand, the group $\Gamma = \mathrm{SL}_2\left(\mathbb{Z}\left[ \frac{1}{p} \right] \right)$ ($p$ a prime), which is an amalgamated product \cite[Ch. 2, \S 1.4, Corollary 2]{Serre-Trees} does have bounded generation  \cite{MRS}, \cite{Vsem}.  (More precisely, in this case $\Gamma$ admits a presentation $\Gamma = \Gamma_1 *_{\Gamma_0} \Gamma_2$ where both $\Gamma_1$ and $\Gamma_2$ are isomorphic to $\mathrm{SL}_2(\mathbb{Z})$, and $\Gamma_0$ is identified with the subgroup $\Gamma_0(p)$ consisting of matrices $\left( \begin{array}{cc} a & b \\ c & d \end{array} \right)$ satisfying $c \equiv 0 (\mathrm{mod}\: p)$, so the fact that $\vert \Gamma_0 \backslash \Gamma_i / \Gamma_0 \vert = 2$ for $i = 1, 2$ follows from the Bruhat decomposition for $\mathrm{SL}_2$ over $\mathbb{F}_p = \mathbb{Z}/p\mathbb{Z}$.) Nevertheless, if $\mathbb{H}$ is the division algebra of usual quaternions (corresponding to the pair $(-1 , -1)$) over $\mathbb{Q}$ and $G = \mathrm{SL}_{1 , \mathbb{H}}$ is the associated norm 1 group, then by Theorem \ref{zdlx} the group $\Delta = G(R)$ of points over the ring $R = \mathbb{Z}\left[ \frac{1}{p} , \frac{1}{q} \right]$, where $p$ and $q$ are distinct odd primes, does not have bounded generation. Surprisingly, this happens despite the fact that $\Delta$ shares many
group-theoretic properties with $\Gamma$, viz. $\Delta$ has a presentation $\Delta_1 *_{\Delta_0} \Delta_2$ where $\Delta_1 , \Delta_2$ are virtually free groups with $\Delta_0$ having a description very similar to that of $\Gamma_0$; both $\Gamma$ and $\Delta$ do not have noncentral normal subgroups of infinite index; both $\Gamma$ and $\Delta$ are $SS$-rigid (in fact, super-rigid). The true reason behind the fundamental distinction between $\Gamma$ and $\Delta$ as far as bounded generation is concerned remains elusive at this point. So, we would like to propose the following.

\vskip2mm

\noindent {\bf Problem I.} {\it Give a criterion or at least a verifiable and general enough sufficient condition for a free amalgamated product $\Gamma = \Gamma_1 *_{\Gamma_0} \Gamma_2$ to have bounded generation.}

\vskip2mm

Of course, one is particularly interested in a sufficient condition that would explain bounded generation of  $\mathrm{SL}_2\left(\mathbb{Z}\left[ \frac{1}{p} \right] \right)$ from the group-theoretic perspective, so the case where $\Gamma_1$ and $\Gamma_2$ are (virtually) free groups is of special significance. On the other hand, our Theorem \ref{zdlx} seems to suggest that when $\Gamma_1$ and $\Gamma_2$ are both surface groups, and a common subgroup $\Gamma_0 \neq \Gamma_1, \Gamma_2$ has index $\geq 3$ in at least one of them, the amalgamated product $\Gamma = \Gamma_1 *_{\Gamma_0} \Gamma_2$ never has bounded generation.

\vskip1mm

Speaking about linear groups, one can ask if all linear groups over a field of characteristic zero that have bounded generation can be obtained by some natural operations from $S$-arithmetic groups. A more specific question is whether a linear group which is a nontrivial amalgamated product and has bounded generation at the same time must be $S$-arithmetic?

\vskip1mm

As we have already mentioned in the introduction, bounded generation of $S$-arithmetic subgroups of absolutely almost simple algebraic groups, under some natural assumptions, implies the congruence subgroup property. So, the study of bounded generation in this context was seen as a new approach to Serre's Congruence Subgroup Conjecture \cite{Serre-CSP}, particularly for anisotropic groups where one cannot use unipotent elements. Since now we know that infinite $S$-arithmetic subgroups of anisotropic absolutely almost simple algebraic groups are not boundedly generated, one may wonder about the possible ways to modify this strategy. In this regard, we would like to point out that it was shown in \cite{PR-CSP} that under the same natural assumptions, the congruence subgroup property follows from a weaker property of polynomial index growth for an $S$-arithmetic subgroup $\Gamma$:

\vskip2mm

\noindent (PIG) \parbox[t]{15cm}{there exists positive constants $c$ and $k$ such that for any integer $n > 0$ the subgroup $\Gamma^{(n)}$ generated by the $n$-th powers of elements of $\Gamma$ has finite index in $\Gamma$ bounded by $cn^k$,}

\vskip2mm

\noindent or even its weaker version

\vskip2mm

\noindent (PIG)$'$ \parbox[t]{15cm}{for any $n > 0$, the subgroup $\Gamma^{(n)}$ is of finite index in $\Gamma$ and for fixed $n$ and a prime $p$ there exist
$c , k > 0$ such that $[\Gamma : \Gamma^{(np^{\alpha})}] \leq cp^{k\alpha}$ for all $\alpha > 0$.}

\vskip2mm

\noindent Furthermore, it was shown in \cite[\S 7]{PR-CSP} that condition (PIG)$'$ for $\mathrm{SL}_n(\mathbb{Z})$, $n \geq 3$, can be verified by quite straightforward computations that rely only on the well-known commutator identities for elementary matrices. So,  it may be realistic to verify (PIG)$'$ in other situations.

\vskip2mm

\noindent {\bf Problem II.} {\it Let $G$ be an absolutely almost simple algebraic group over a number field $K$, and $S$ be a finite set of places of $K$ containing all archimedean ones. Prove that if the $S$-rank $\displaystyle \mathrm{rk}_S\: G := \sum_{v \in S} \mathrm{rk}_{K_v}\: G$ is $\geq 2$, then $S$-arithmetic subgroups of $G$ satisfy condition \rm{(PIG)}$'$.}

\vskip2mm

\noindent We recall that according to Margulis' Normal Subgroup Theorem \cite[Ch. VIII]{margulis}, any noncentral normal subgroup of a higher rank $S$-arithmetic subgroup $\Gamma$ as above, hence in particular the subgroup $\Gamma^{(n)}$ for any $n > 0$, automatically has finite index.

\vskip1mm

We  conclude with one more problem which is at the meeting ground of Problems I and II.

\vskip2mm

\noindent {\bf Problem III.} {\it Give a criterion or at least a verifiable sufficient condition for a free amalgamated product $\Gamma = \Gamma_1 *_{\Gamma_0} \Gamma_2$ to satisfy condition \rm{(PIG)}$'$.}

\vskip2mm

We note that a number of results relating (BG), (PIG) and some other finiteness conditions in various situations were obtained in \cite{PySe}.

%\vskip5mm

\vskip2mm

\noindent {{\small {\bf Acknowledgements.} We are grateful to H.~Abels,  D.~Segal, G.~Tomanov and the two anonymous referees for their comments and suggestions that helped to improve the exposition.}}

\bibliographystyle{amsplain}

\end{document}